\documentclass{imsart}

\RequirePackage[OT1]{fontenc}
\RequirePackage{amsthm,amsmath}
\RequirePackage[numbers]{natbib}
\RequirePackage[colorlinks,citecolor=blue,urlcolor=blue]{hyperref}
\usepackage{lineno,hyperref}
\usepackage{comment}
\usepackage{amsmath,amsxtra,amssymb,latexsym, amscd}
\usepackage[mathscr]{eucal}
\usepackage{amsthm}

\startlocaldefs
\numberwithin{equation}{section}
\theoremstyle{plain}
\newtheorem{theorem}{Theorem}[section]

\newtheorem{lemma}[theorem]{Lemma}
\newtheorem{proposition}{Proposition}

\theoremstyle{definition}
\newtheorem{definition}[theorem]{Definition}
\newtheorem{remark}{Remark}

\endlocaldefs

\allowdisplaybreaks

\begin{document}

\begin{frontmatter}
\title{Regularity of solutions of abstract linear evolution equations}
\runtitle{Linear evolution equations}
\thankstext{T1}{This work was supported by JSPS KAKENHI Grant Number 20140047.}

\begin{aug}
\author{Vi$\hat{\d{e}}$t T$\hat{\text{o}}$n T\d{a}\thanksref{T1} \ead[label=e1]{taviet.ton[at]ist.osaka-u.ac.jp}}

\address{Department of Information and Physical Sciences\\
  Graduate School of Information Science and Technology, Osaka University\\
  Suita, Osaka 565-0871, Japan\\
\printead{e1}}

\runauthor{V. T. T\d{a}}

\affiliation{Osaka University}

\end{aug}

\begin{abstract}
In this paper, we study regularity of solutions to linear evolution equations of the form $dX+AXdt=F(t)dt$ in a Banach space $H$, where $A$ is a sectorial operator  in $H$ and $A^{-\alpha} F \, (\alpha>0)$  belongs to a weighted H\"{o}lder continuous function space. Similar results are obtained for linear evolution equations with additive noise of the form $dX+AXdt=F(t)dt+G(t)dW(t)$ in a separable Hilbert space $H$, where $W(t)$ is a  cylindrical Wiener process. Our results are applied to  a model arising in neurophysiology, which has been proposed by Walsh \cite{Walsh}.
\end{abstract}


\tableofcontents
\end{frontmatter}

\section{Introduction and main results}

In recent years, existence, uniqueness and regularity of solutions to stochastic partial differential equations have been extensively studied by many authors. These topics have been developed mainly by using three approaches, that is, the semigroup methods (see   Da Prato-Zabczyk \cite{prato},  T\d{a}-Yagi \cite{Ton2}, T\d{a} \cite{Ton3},  and references therein),  the variational methods (see Rozovskii \cite{Rozovskii}, Pr\'{e}v\^{o}t-R\"{o}ckner \cite{PrevotRockner}, and references therein), and the martingale measure methods (see Walsh \cite{Walsh}). 
Among others, stochastic linear evolution equations have been studied by many authors in Hilbert or $L_2$ spaces (see  Rozovskii \cite{Rozovskii}, 
Da Prato et al. \cite{prato0,prato}, T\d{a}  \cite{Ton1}), in weighted Sobolev $L_p$ spaces (see  Krylov-Lototsky \cite{Krylov1}),  in weighted H\"{o}lder spaces (see Mikulevicius \cite{Mikulevicius2}), and in M-type 2 Banach space (Brze\'{z}niak \cite{Brzezniak}, T\d{a}-Yagi \cite{Ton0}, T\d{a}-Yamamoto-Yagi \cite{Ton4}). 
 
In this paper, we shall study regularity of solutions to both deterministic and stochastic linear evolution equations whose coefficients belong to  weighted H\"{o}lder  continuous function spaces by using the semigroup methods. Our results  can be applied to a class of  stochastic partial differential equations such as the Zakai equation in the nonlinear filtering problem (see  Rozovskii \cite{Rozovskii}), the stochastic heat equation (see Hairer \cite{Hairer}), Nagumo's equation, Hodgkin-Huxley equations (see \cite{Walsh,prato}), which can be approximated by the equation:
$$\frac{\partial X}{\partial t}=\Delta X -aX + F(t) +G(t)\dot W.$$
As a consequence, some results obtained by Walsh  \cite{Walsh} for a model arising in neurophysiology are improved.

Let us  introduce the framework in the deterministic case.
 We consider the Cauchy problem for a linear evolution equation
\begin{equation} \label{P1}
\begin{cases}
dX+AXdt=F(t)dt,\hspace{1cm} 0<t\leq T,\\
X(0)=\xi
\end{cases}
\end{equation}
in   a Banach space $H$ with norm $\|\cdot\|,$
where $T>0$,  $A\colon\mathcal D(A)\subset H\to H$ is a densely defined, closed linear operator in $H$,
 and $F$ is an $H$-valued addition external force. 
\begin{definition}
Let $F$ satisfy the condition
$$\int_0^t S(t-s) F(s)ds<\infty,  \hspace{2cm} t\in [0,T],$$
where $S(t)$ is the $C_0$-semigroup  in $H$ generated by  the operator $(-A)$. 
Then the process $X$ defined by 
$$X(t)=S(t)\xi + \int_0^t S(t-s) F(s)ds, \hspace{2cm} t\in [0,T]$$
is called a mild solution of \eqref{P1}.
\end{definition}
(In Remark \ref{remark2} below, we will give a discussion about a relation between weak, mild and strong solutions.)

When the initial value $\xi\in H$ is arbitrary and  $F\in L^1([0,T];H)$, the Cauchy problem   \eqref{P1} possesses a unique continuous mild solution.  
Meanwhile, when $\xi\in \mathcal D(A)$ and $F\in W^{1,p}([0,T];H) $ with $p\geq 1$,  the mild solution possesses the regularity
$$X\in \mathcal C^1([0,T];H)\cap \mathcal C([0,T]; \mathcal D(A))\cap W^{1,p}([0,T];H),$$
where $W^{1,p}([0,T];H)$ denotes  the Sobolev space of all $H$-valued functions $u$ on $[0,T]$ whose weak derivative $u'$ belongs to $L^p([0,T];H)$ (see e.g., \cite{Ball,prato}).

When $(-A)$ is a sectorial operator (i.e. it satisfies conditions (H1) and (H2) below), it generates an analytical semigroup $S(t)$.  In  \cite{Pazy}, the author considered the case $\xi=0$ and $F\in \mathcal C^\alpha([0,T];H)$ for some $\alpha \in (0,1),$ where $\mathcal C^\alpha([0,T];H)$ denotes the space of  $\alpha$-H\"{o}lder continuous functions on $[0,T]$. Then the mild solution of \eqref{P1} satisfies
$$X\in \mathcal C^1([0,T];H)\cap \mathcal C([0,T];\mathcal D(A)).$$
 In  \cite{Sinestrari}, under the assumption  $F(0)\in \mathcal D_A(\alpha,\infty)$, it has been verified that 
$$X\in \mathcal C^{1,\alpha}([0,T];H)\cap \mathcal C^\alpha([0,T];\mathcal D(A)),$$
where $\mathcal C^{1,\alpha}$ is the space of $\mathcal C^1$ functions with derivatives in $\mathcal C^\alpha.$ 
If $F$ is taken from $ \mathcal C([0,T];\mathcal D_A(\alpha,\infty)) $ then (see \cite{prato-2})
$$X\in \mathcal C^1([0,T];\mathcal D_A(\alpha,\infty))\cap \mathcal C([0,T];\mathcal D_A(\alpha+1,\infty)),$$
where $ \mathcal D_A(\alpha,\infty)$ is the space of all $x\in H$ such that 
$\sup_{t>0} \frac{\|S(t)x-x\|}{t^\alpha}<\infty$,  and $\mathcal D_A(\alpha+1,\infty)=\{x\in \mathcal D(A)\colon Ax\in  \mathcal D_A(\alpha,\infty)\}$.
When $\xi\in H$ is arbitrary and $F$ belongs to a weighted H\"{o}lder continuous function space $\mathcal F^{\beta, \sigma}((0,T]; H)$ (see the definition below), Yagi \cite{yagi0} showed that
$$X\in \mathcal C^1((0,T];H)\cap \mathcal C([0,T];H) \cap \mathcal C((0,T];\mathcal D(A)).$$
He also obtained the maximal regularity for both initial value $\xi \in \mathcal D(A^\beta)$ and external force function $F\in \mathcal F^{\beta, \sigma}((0,T]; H)$  (see \cite{yagi1}-\cite{yagi}):
$$A^\beta X\in \mathcal  C([0,T];H),$$
$$\frac{dX}{dt}, AX \in \mathcal F^{\beta, \sigma}((0,T]; H).$$

In the present paper,  we assume that $A^{-\alpha}F$ belongs to the weighted H\"{o}lder continuous function space $\mathcal F^{\beta, \sigma}((0,T]; H)$ with some positive $\alpha$, i.e.  $F$ satisfies both temporal  and spatial regularity. We will show both temporal  and spatial regularity of solutions to  \eqref{P1}.  Similar results are obtained for linear evolution equations in Hilbert spaces with additive noise. 
Our result will be applied to a model arising in neurophysiology.

Let us review the function space  $\mathcal F^{\beta, \sigma}((0,T]; H)$ 
 for two exponents $0<\sigma<\beta\leq 1$ (see  \cite{yagi}).  The space  $\mathcal F^{\beta, \sigma}((0,T]; H)$ consists of all $H$-valued continuous functions $f(t)$ on $(0,T]$ (resp. $[0,T]$) when $0<\beta<1$ (resp. $\beta=1$)  with the following three properties:
\begin{itemize}
  \item [\rm (i)] When $\beta<1$, 
  \begin{equation}  \label{P2}
  t^{1-\beta} f(t)  \text{  has a limit as } t\to 0.
  \end{equation}
  \item [\rm (ii)] The function $f$ is H\"{o}lder continuous with the exponent $\sigma$ and with the weight $s^{1-\beta+\sigma}$, i.e.
  \begin{equation} \label{P3}
\begin{aligned}
&\sup_{0\leq s<t\leq T} \frac{s^{1-\beta+\sigma}\|f(t)-f(s)\|}{(t-s)^\sigma}\\
&=\sup_{0\leq t\leq T}\sup_{0\leq s<t}\frac{s^{1-\beta+\sigma}\|f(t)-f(s)\|}{(t-s)^\sigma}<\infty.
\end{aligned}
\end{equation}
  \item [\rm (iii)] 
   \begin{equation} \label{P4}
   \begin{aligned}
  \lim_{t\to 0} & w_f(t)=0, \\
& \text{ where } w_f(t)=\sup_{0\leq s  <t}\frac{s^{1-\beta+\sigma}\|f(t)-f(s)\|}{(t-s)^\sigma}.
\end{aligned}
  \end{equation}
  \end{itemize}  
Then  $\mathcal F^{\beta, \sigma}((0,T]; H)$ becomes a Banach space with  norm
$$\|f\|_{\mathcal F^{\beta, \sigma}((0,T]; H)}=\sup_{0\leq t\leq T} t^{1-\beta} \|f(t)\|+ \sup_{0\leq s<t\leq T} \frac{s^{1-\beta+\sigma}\|f(t)-f(s)\|}{(t-s)^\sigma}.$$
 For simplicity, if not specified the norm of $f$ in $\mathcal F^{\beta, \sigma}((0,T];H)$ will be denoted by $\|f\|_{\mathcal F^{\beta, \sigma}}$. 
The following useful inequality follows the definition directly. For every $ f\in \mathcal F^{\beta, \sigma}((0,T]; H), 0<s<t\leq T$ we have
\begin{equation} \label{P5}
\begin{cases}
\|f(t)\|\leq \|f\|_{\mathcal F^{\beta, \sigma}} t^{\beta-1}, \\
 \|f(t)-f(s)\| \leq w_f(t) (t-s)^{\sigma} s^{\beta-\sigma-1}\leq \|f\|_{\mathcal F^{\beta, \sigma}} (t-s)^{\sigma} s^{\beta-\sigma-1}.
\end{cases}
\end{equation}
In addition, it is not hard to show that
\begin{equation} \label{P6}
\mathcal F^{\gamma,\sigma} ((0,T];H)\subset \mathcal F^{\beta,\sigma} ((0,T];H), \hspace{2cm} 0<\sigma<\beta<\gamma\leq 1.
\end{equation}
\begin{remark}
The space $\mathcal F^{\beta, \sigma}((0,T]; H)$ is not  a trivial space.  When  $0<\sigma<\beta<1$,  $f(t) =t^{\beta-1} g(t) \in \mathcal F^{\beta, \sigma}((0,T]; H),$ where  $g(t)$ is any $H$-valued  function on $[0,T]$ such that
 $g\in \mathcal C^\sigma([0,T];H)$
  and $g(0)=0.$
When  $0<\sigma<\beta=1$,  the space $ \mathcal F^{1, \sigma}((0,T]; H)$ includes the space of  all H\"{o}lder continuous functions with  the exponent $\sigma.$ 
\end{remark}

Let us now formulate the precise conditions on the coefficients in \eqref{P1}. 
\begin{itemize}
  \item [(\rm{H1})] The spectrum $\sigma(A)$ of $A$ is  contained in an open sectorial domain $\Sigma_{\varpi}$: 
\begin{equation*} 
\sigma(A) \subset  \Sigma_{\varpi}=\{\lambda \in \mathbb C: |\arg \lambda|<\varpi\}, \quad \quad 0<\varpi<\frac{\pi}{2}.
       \end{equation*}
  \item [(\rm{H2})] The resolvent satisfies the estimate  
\begin{equation*} \label{H2}
          \|(\lambda-A)^{-1}\| \leq \frac{M_{\varpi}}{|\lambda|}, \quad\quad\quad \quad   \lambda \notin \Sigma_{\varpi}
     \end{equation*}
     with some  constant $M_{\varpi}>0$ depending only on the angle $\varpi$.
  \item [(\rm{H3})] 
\begin{equation*}  \label{H6}
\begin{aligned}
       & A^{-\alpha}F\in \mathcal F^{\beta, \sigma}((0,T];H)  \hspace{0.5cm}  \text{  with } 0<\sigma< \beta\leq 1 \text{  and  } \\
    & \frac{1+\sigma}{4}<\alpha\leq \frac{\beta}{2}.
\end{aligned}
              \end{equation*}
     \end{itemize} 

Under the assumptions (H1) and (H2), it is well-known that
\begin{proposition}[see e.g., \cite{yagi}]
Let {\rm (H1)} and {\rm (H2)}  be satisfied. Then
\begin{itemize}
\item  [\rm (i)] $(-A)$ generates a semigroup $S(t)=e^{-tA}$, i.e. $S(t-s)$ enjoys the property
\begin{equation*}
\begin{aligned}
\begin{cases}
S(t+s)=S(t)S(s), &\hspace{1cm} 0\leq  t, s<\infty.\\
S(0)=I. &
\end{cases}
\end{aligned}
\end{equation*}
\item  [\rm (ii)] For every $t>0$ and $\theta\geq 0$
\begin{equation} \label{P7}
\begin{aligned}
& \|A^\theta S(t)\| \leq \iota_\theta t^{-\theta}, \\
& \text{ where  } \iota_\theta:=\sup_{0\leq  t<\infty} t^\theta \|A^\theta S(t)\|<\infty.
\end{aligned}
\end{equation}
  In particular, 
\begin{equation} \label{P8}
\|S(t)\|\leq \iota_0, \hspace{2cm}  0\leq  t<\infty.
\end{equation}
\item  [\rm (iii)] For every $\theta>0$ there exists a constant $\upsilon_\theta$ such that
   \begin{equation}  \label{P9}
   \|A^{-\theta}\|\leq \upsilon_\theta.
   \end{equation}
\item[\rm (iv)] For every $0< \theta\leq 1 $
 \begin{equation} \label{P10}
 t^\theta A^{\theta} S(t)  \text{ converges to } 0 \text{ strongly on } H \text { as } t\to 0.
 \end{equation}
\end{itemize}
\end{proposition}
From now when we use notations $\iota_\theta$ or $\upsilon_\theta$, it refers to constants  in \eqref{P7} or \eqref{P9}, respectively. 
 Now we can state the main result for \eqref{P1}.
\begin{theorem}\label{theorem1}  
Let {\rm (H1)}, {\rm (H2)} and {\rm (H3)} be satisfied. Let  $\xi$  take  values in $\mathcal D(A^\beta).$  Then there exists a unique  mild solution of \eqref{P1} possessing the regularity:
$$X\in \mathcal C((0,T];\mathcal D(A^{1-\alpha})), $$
$$ A^{\alpha} X\in \mathcal C([0,T];H)\cap \mathcal F^{\beta-\sigma+\gamma,\gamma}((0,T];H)  \hspace{1cm} \text{ for every }\gamma\in [0, \sigma].$$
Furthermore, $X$ satisfies the estimate
\begin{equation} \label{P11}
\|X(t)\|\leq \iota_\alpha B(\beta,1-\alpha) \|A^{-\alpha}F\|_{\mathcal F^{\beta,\sigma}} t^{\beta-\alpha} +\iota_0 \|\xi\|, \hspace{1cm} t\in [0,T],
\end{equation}
where $B(\cdot,\cdot)$ is the beta function.
\end{theorem}

Let us next formulate the analogous result for linear evolution equations with additive  noise.
We proceed to study the Cauchy problem for an abstract stochastic evolution  equation
\begin{equation} \label{P12}
\begin{cases}
dX+AXdt=F(t)dt+ G(t)dW(t),\hspace{1cm} 0<t\leq T,\\
X(0)=\xi
\end{cases}
\end{equation}
in a separable Hilbert space $H$. Here,   $W(t)$  is a cylindrical Wiener process  on  a separable Hilbert space $U$, which is defined on a filtered probability space $(\Omega, \mathcal F,\mathcal F_t,\mathbb P).$ The coefficient $G$ is a measurable process  from $([0,T],\mathcal B([0,T]))$ to $(L_2(U;H),\mathcal B(L_2(U;H))),$  where $L_2(U;H)$ denotes the space of all Hilbert-Schmidt operators from $U$ to $H$. The linear operator $A$ and the function $F$ are the same to ones in \eqref{P1}.  And the initial value $ \xi$ is an $\mathcal  F_0$-measurable  random variable.
\begin{definition}\label{Def1}
Let $F$ and $G$ satisfies the conditions
$$\int_0^T \|S(t-s)F(s)\|ds<\infty$$
and 
$$    \int_0^T \|S(t-s)G(s)\|_{L_2(U;H)}^2ds<\infty.$$ 
Then an $H$-valued predictable  process $X(t), t\in [0,T]$ defined by 
\begin{align*}
X(t)=&S(t)\xi +\int_0^tS(t-s) F(s)ds+ \int_0^t S(t-s) G(s)dW(s) 
\end{align*}
 is called a  mild solution of \eqref{P12}.
\end{definition}  
\begin{remark}  \label{remark2}
In the semigroup approach to parabolic evolution equations,  mild (and \textit{weak}) solutions are mainly studied, because the existence of \textit{strong} solutions  (or \textit{real} solutions) is very rare \cite{prato}. Let us review the notions  of strong and weak solutions. Suppose that $F$ and $G$ satisfy the conditions
$$\int_0^T \|F(s)\|ds<\infty$$
and 
$$    \int_0^T \|G(s)\|_{L_2(U;H)}^2ds<\infty.$$ 
Then an $H$-valued predictable  process $X$ on $ [0,T]$ is called a strong solution to \eqref{P12} if $X$ takes values in $\mathcal D(A)$ a.s. and satisfies 
$$\int_0^T \|AX(s)\|ds<\infty, \hspace{2cm} \text{a.s.}$$
and almost surely 
$$X(t)=\xi+\int_0^t [F(s)-AX(s)]ds+ \int_0^t G(s)dW(s), \hspace{2cm} t\in [0,T].  $$
 The process $X$ is called a weak solution if for every $h\in \mathcal D(A^*)$ and $t\in [0,T]$ we have
\begin{align*}
\langle X(t),h\rangle=&\langle \xi,h\rangle+ \int_0^t [\langle F(s),h\rangle-\langle X(s),A^*h\rangle]ds\\
&+\int_0^t\langle G(s),h\rangle dW(s).
\end{align*}
It is clear that a strong solution is also a weak solution. In addition, a weak solution is a mild solution \cite{prato}. The inverses  are not true in general. They can take place, however, under some special conditions. In \cite{prato}, the authors showed that  a mild solution becomes a weak solution  under some conditions on $G$. In \cite{Ton}, under the assumption 
$$\|AS(t)\| \leq C t^{-\delta}, \hspace{2cm} t>0$$ 
with some constants $C>0$ and $\delta\in (0,\frac{1}{2})$, we proved that a mild solution is also a strong solution. 
\end{remark}
We suppose that the process $G$ in the equation \eqref{P12} belongs to a weighted H\"{o}lder continuous function space:
  \begin{equation*}
\rm{(H4)}  \hspace{2cm}  G\in \mathcal F^{\beta, \sigma} ((0,T];L_2(U;H)) \hspace{1cm}  \text{with }  0<\sigma<\beta-\frac{1}{2}.
  \end{equation*}
Our results for the stochastic case are as follows. In the case there is no external force, we have
\begin{theorem} \label{theorem2}
Assume that $F\equiv 0$ on $[0,T]$.  Let {\rm (H1)}, {\rm (H2)} and {\rm (H4)} be satisfied. Let  $\xi$  take values in $\mathcal D(A^\beta) $ a.s. such that $ \mathbb E\|A^\beta \xi\| <\infty$.  Then there exists a unique mild solution of  \eqref{P12} possessing the regularity:
$$X\in \mathcal C((0,T];\mathcal D(A^\nu)), \quad A^{\alpha_1} X\in  \mathcal C^\gamma([\epsilon,T];H) \hspace{1cm}\text{ a.s.},$$
and 
$$\mathbb E \|A^{\alpha_1} X\|\in  \mathcal F^{\beta,\sigma}((0,T];\mathbb R) $$
  for every $0<\nu<\frac{1}{2}, 0<\alpha_1\leq \frac{1}{2}-\sigma,  0<\gamma<\sigma$ and $\epsilon\in (0,T]$. Furthermore, $X$ satisfies the estimate
  \begin{equation} \label{P13}
  \mathbb E\|X(t)\|\leq C[\mathbb E\|A^\beta\xi\|+\|G\|_{\mathcal F^{\beta, \sigma} ((0,T];L_2(U;H))} t^{\beta-\frac{1}{2}}], \hspace{1cm} t\in [0,T],
  \end{equation}
  where $C$ is some constant depending only on the exponents and  constants $\iota_\theta,$ $ \upsilon_\theta \,(\theta\geq 0)$.
\end{theorem}
In general, we have
\begin{theorem} \label{theorem3}
 Let {\rm (H1)}, {\rm (H2)},  {\rm (H3)} and {\rm (H4)} be satisfied. Let  $\xi$  take values in $\mathcal D(A^\beta) $ a.s.  such that $ \mathbb E\|A^\beta \xi\| <\infty$.    If $\alpha\leq \frac{1}{2}-\sigma$ 
 then there exists a unique  mild solution of \eqref{P12} possessing the regularity:
$$X\in \mathcal C((0,T];\mathcal D(A^\nu)), \quad A^{\alpha} X\in  \mathcal C^\gamma([\epsilon,T];H) \hspace{1cm}\text{ a.s.},$$
and 
$$\mathbb E \|A^{\alpha} X\|\in  \mathcal F^{\beta,\sigma}((0,T];\mathbb R) $$
  for every $0<\nu<\frac{1}{2},  0<\gamma<\sigma$ and $\epsilon\in (0,T]$. Furthermore,  the estimate
  \begin{equation*} 
  \mathbb E\|X(t)\|\leq C[\mathbb E\|A^\beta\xi\|+\|A^{-\alpha}F\|_{F^{\beta,\sigma}} t^{\beta-\alpha}+\|G\|_{\mathcal F^{\beta, \sigma} ((0,T];L_2(U;H))} t^{\beta-\frac{1}{2}}] 
  \end{equation*}
  holds true for every $t\in [0,T]$, where $C$ is some constant depending only on the exponents and  constants $\iota_\theta, \upsilon_\theta \,(\theta\geq 0)$.
\end{theorem}
\begin{remark} 
According to  the proof of the above theorems in the next section, we can see that the results in this paper hold true for not only non-random functions $F(t)$ and $ G(t)$ but also  random ones, i.e. for $F$ and $ G$ depending on both $t\in [0,T]$ and $\omega\in \Omega$ and satisfying {\rm (H3)} and {\rm  (H4)} almost surely. In that case, we shall need more conditions on $F$ and $G$ such as
\begin{itemize}
  \item  [{\rm (i)}] $F$ and $G$ are measurable with respect to $(t,\omega).$ 
  \item [{\rm (ii)}] $\mathbb E\|G\|_{\mathcal F^{\beta, \sigma}((0,T];L_2(U;H))}^2 <\infty.$
\end{itemize}
\end{remark}
The rest of the present paper is organized as follows. In the next section, we shall present the proofs of Theorems \ref{theorem1}, \ref{theorem2} and \ref{theorem3}.  In Section \ref{section3}, we shall apply the main results to a model arising in neurophysiology.
\section{Proof of the main theorems}   \label{section2}
\begin{proof}[Proof of Theorem \ref{theorem1}]
 Let us show that \eqref{P1} possesses a unique mild solution in the space
 $$X\in \mathcal C((0,T];\mathcal D(A^{1-\alpha})), $$
 which satisfies the estimate \eqref{P11}. By  \eqref{P5}, \eqref{P7} and {\rm (H3)},  we have
\begin{equation} \label{P14}
\begin{aligned}
\int_0^t\|S(t-s) F(s)\|ds&=\int_0^t\|A^\alpha S(t-s) A^{-\alpha} F(s)\|ds\\
&\leq \iota_\alpha \|A^{-\alpha}F\|_{\mathcal F^{\beta,\sigma}} \int_0^t  (t-s)^{-\alpha}s^{\beta-1}ds\\
& =\iota_\alpha \|A^{-\alpha}F\|_{\mathcal F^{\beta,\sigma}} B(\beta,1-\alpha)t^{\beta-\alpha} <\infty, \hspace{1cm} t\in [0,T].
\end{aligned}
\end{equation}
Then the integral $\int_0^tS(t-s) F(s)ds$ is well-defined. Consequently, 
\eqref{P1} possesses a  mild solution given by
$$X(t)=S(t)\xi +\int_0^tS(t-s) F(s)ds$$
(the uniqueness is obvious). 
On the other hand, 
\begin{align}
&\int_0^t A^{1-\alpha}S(t-s) F(s)ds  \notag\\
&=\int_0^tA S(t-s) [A^{-\alpha} F(s)-A^{-\alpha} F(t)]ds+ \int_0^tA S(t-s)  dsA^{-\alpha} F(t)  \notag\\
&=\int_0^tA S(t-s) [A^{-\alpha} F(s)-A^{-\alpha} F(t)]ds+ [I-S(t)]A^{-\alpha} F(t).   \label{P15}
\end{align}
Thanks to \eqref{P5}, \eqref{P7} and {\rm (H3)},
\begin{align*}
&\int_0^t\|A S(t-s) [A^{-\alpha} F(s)-A^{-\alpha} F(t)]\|ds\\
&\leq \int_0^t\|A S(t-s)\|  \|A^{-\alpha} F(s)-A^{-\alpha} F(t)\|ds\\
&\leq \iota_1 \|A^{-\alpha}F\|_{\mathcal F^{\beta,\sigma}} \int_0^t  (t-s)^{\sigma-1}s^{\beta-\sigma-1}ds\\
& =\iota_1 \|A^{-\alpha}F\|_{\mathcal F^{\beta,\sigma}} B(\beta-\sigma,\sigma)t^{\beta-1} <\infty, \hspace{1cm} t\in (0,T].
\end{align*}
Hence, it is easy to see that $\int_0^t AS(t-s) [A^{-\alpha} F(s)-A^{-\alpha} F(t)]ds$ is continuous in $t\in (0,T].$
In addition, since $A$ is closed, we obtain
$$A\int_0^t S(t-s) [A^{-\alpha} F(s)-A^{-\alpha} F(t)]ds=\int_0^tA S(t-s) [A^{-\alpha} F(s)-A^{-\alpha} F(t)]ds.$$
 Therefore, by \eqref{P15}, it is seen that 
$$A^{1-\alpha} \int_0^t S(t-s) F(s)ds=\int_0^t A^{1-\alpha}S(t-s) F(s)ds$$
and $A^{1-\alpha} \int_0^t S(t-s) F(s)ds$ is continuous in $t\in (0,T]$. As a consequence,
$$A^{1-\alpha}X(t)=A^{1-\alpha}S(t)\xi+A^{1-\alpha} \int_0^t S(t-s) F(s)ds$$
belongs to $\mathcal C((0,T];H)$. Furthermore, by  \eqref{P8} and \eqref{P14}, 
$X$ satisfies the estimate \eqref{P11}.

Let us next prove that $A^{\alpha} X(t)=A^{\alpha} S(t)\xi +A^{\alpha} \int_0^tS(t-s) F(s)ds$ is continuous in $t\in [0,T]$. 
Indeed, the first term in the right-hand side of the latter equality is continuous on $[0,T]$, because $\alpha<\beta$ and
$A^{\alpha} S(t)\xi=A^{\alpha-\beta} S(t)A^\beta \xi.$  So, it suffices  to show that  the second term is also  continuous on $[0,T]$. 
 Similarly to \eqref{P14}, we have
\begin{equation*} 
\begin{aligned}
\int_0^t\|A^\alpha S(t-s) F(s)\|ds\leq 
\iota_{2\alpha} \|A^{-\alpha}F\|_{\mathcal F^{\beta,\sigma}} B(\beta,1-2\alpha)t^{\beta-2\alpha} <\infty, \hspace{0.2cm} t\in [0,T].
\end{aligned}
\end{equation*}
In addition, since $A^{\alpha}$ is closed, we obtain 
$$A^{\alpha} \int_0^tS(t-s) F(s)ds=\int_0^tA^{\alpha} S(t-s) F(s)ds.$$
Hence, it is easy to see that    $A^{\alpha} \int_0^tS(t-s) F(s)ds$ is  continuous in $t\in [0,T].$

Let us now verify that $A^{\alpha} X\in \mathcal F^{\beta-\sigma+\gamma,\gamma}((0,T];H)$ for every $\gamma\in [0,\sigma]$. We  use the expression
\begin{align*}
A^{\alpha} X(t)=& A^{\alpha}S(t) \xi+\int_0^t A^{2\alpha} S(t-s)[A^{-\alpha}F(s)-A^{-\alpha}F(t)]ds\\
&+\int_0^t A^{2\alpha} S(t-s)dsA^{-\alpha}F(t)\\
=& A^{\alpha}S(t) \xi+\int_0^t A^{2\alpha} S(t-s)[A^{-\alpha}F(s)-A^{-\alpha}F(t)]ds\\
&+A^{2\alpha-1}[I-S(t)]A^{\alpha-1} F(t)\\
=&J_1(t)+J_2(t)+J_3(t).
\end{align*}
We will show that $J_1, J_2$ and $ J_3$ belong to  $\mathcal F^{\beta-\sigma+\gamma,\gamma}((0,T];H)$.

{\it Proof for $J_1$}. We prove that 
$$J_1\in  \mathcal F^{\beta,\gamma}((0,T];H) \subset \mathcal F^{\beta-\sigma+\gamma,\gamma}((0,T];H)\hspace{2cm} (\text{see }\eqref{P6}). $$
Indeed, applying \eqref{P10} to  
$$t^{1-\beta} A^{\alpha}S(t) \xi=A^{\alpha-1} t^{1-\beta} A^{1-\beta}S(t) A^\beta \xi,$$
 we obtain
$\lim_{t\to 0} t^{1-\beta} A^{\alpha}S(t) \xi=0.$
Hence, the condition \eqref{P2} is fulfilled.

On the other hand,  by using \eqref{P7},  for $0<s<t\leq T$ we have
\begin{align}
&\frac{s^{1-\beta+\gamma} \|A^\alpha[S(t)-S(s)] \xi\|}{(t-s)^\gamma} \notag\\
&\leq \frac{\|A^{\alpha-\gamma-1}[S(t-s)-I]\|}{(t-s)^\gamma}  s^{1-\beta+\gamma}\|A^{1-\beta+\gamma}S(s) A^\beta \xi\|\notag\\
&\leq  \frac{\|\int_0^{t-s} A^{\alpha-\gamma}S(u)du\|}{(t-s)^\gamma} \iota_{1-\beta+\gamma} \|A^\beta \xi\|.  \label{P16}
\end{align}
If $\alpha\geq \gamma$ then
\begin{align}
\frac{s^{1-\beta+\gamma} \|A^\alpha[S(t)-S(s)] \xi\|}{(t-s)^\gamma}&\leq  \frac{ \int_0^{t-s} \iota_{\alpha-\gamma} u^{\gamma-\alpha}du}{(t-s)^\gamma} \iota_{1-\beta+\gamma} \|A^\beta \xi\|\notag\\
&=  \frac{\iota_{\alpha-\gamma} \iota_{1-\beta+\gamma} \|A^\beta \xi\| }{1+\gamma-\alpha}(t-s)^{1-\alpha}. \label{P17}
\end{align}
If $\alpha< \gamma$ then by \eqref{P8} and \eqref{P9}
\begin{align}
\frac{s^{1-\beta+\gamma} \|A^\alpha[S(t)-S(s)] \xi\|}{(t-s)^\gamma}&\leq  \frac{ \int_0^{t-s} \upsilon_{\gamma-\alpha} \iota_0du}{(t-s)^\gamma} \iota_{1-\beta+\gamma} \|A^\beta \xi\|\notag\\
&=  \upsilon_{\gamma-\alpha} \iota_0\iota_{1-\beta+\gamma} \|A^\beta \xi\| (t-s)^{1-\gamma}. \label{P18}
\end{align}
Hence,
$$\sup_{0\leq s<t\leq T}\frac{s^{1-\beta+\gamma} \|A^\alpha[S(t)-S(s)] \xi\|}{(t-s)^\gamma}<\infty$$
and
\begin{align*}
&\lim_{t\to 0} \sup_{0<s<t}\frac{s^{1-\beta+\gamma} \|A^\alpha[S(t)-S(s)] \xi\|}{(t-s)^\gamma}=0.
\end{align*}
Therefore, the conditions \eqref{P3} and  \eqref{P4} are  also satisfied. We have thus verified that $J_1\in  \mathcal F^{\beta,\gamma}((0,T];H). $

{\it Proof for $J_2$}. Using \eqref{P5} and {\rm (H3)} we have 
\begin{align*}
\|J_2(t)\|&\leq \int_0^t \|A^{2\alpha} S(t-s)\|\|A^{-\alpha}F(s)-A^{-\alpha}F(t)\|ds\notag\\
&\leq \iota_{2\alpha} w_{A^{-\alpha}F}(t)\int_0^t (t-s)^{\sigma-2\alpha} s^{\beta-\sigma-1}ds\notag\\
&=\iota_{2\alpha} w_{A^{-\alpha}F}(t) B(\beta-\sigma,1+\sigma-2\alpha)t^{\beta-2\alpha}, 
\end{align*}
here
 $$
  w_{A^{-\alpha}F}(t)=\sup_{0\leq s  <t}\frac{s^{1-\beta+\sigma}\|A^{-\alpha}F(t)-A^{-\alpha}F(s)\|}{(t-s)^\sigma} \hspace{1cm} (\text{see  } \eqref{P4}).
$$
This implies that 
\begin{equation} \label{P19}
\lim_{t\to 0} t^{1-(\beta-\sigma+\gamma)} J_2(t)=0  \hspace{1cm} \text{ in } H.
\end{equation} 

We next observe that for $0<s<t\leq T$
\begin{align}
J_2(t)-J_2(s)=&\int_s^t A^{2\alpha} S(t-u)[A^{-\alpha} F(u)-A^{-\alpha} F(t)] du   \notag\\
&+ [S(t-s) -I] \int_0^s A^{2\alpha} S(s-u)[A^{-\alpha} F(u)-A^{-\alpha} F(s)] du \notag\\
&+\int_0^s A^{2\alpha} S(t-u)[A^{-\alpha} F(s)-A^{-\alpha} F(t)] du\notag\\
=&J_{21}(t,s)+J_{22}(t,s)+J_{23}(t,s).    \label{E19}
\end{align}

For $J_{21}(t,s)$, by \eqref{P5} and \eqref{P7} we have the estimate
\begin{align}
\|J_{21}(t,s)\|\leq & \int_s^t \|A^{2\alpha} S(t-u)\| \|A^{-\alpha} F(u)-A^{-\alpha} F(t)\| du   \notag\\
\leq &\int_s^t \iota_{2\alpha} w_{A^{-\alpha} F}(t)(t-u)^{\sigma-2\alpha}   u^{\beta-\sigma-1}du\notag\\
\leq &\iota_{2\alpha} w_{A^{-\alpha} F}(t)s^{\beta-\sigma-1}\int_s^t  (t-u)^{\sigma-2\alpha}   du\notag\\
=&\frac{\iota_{2\alpha} w_{A^{-\alpha} F}(t)s^{(\beta-\sigma+\gamma)-\gamma-1} (t-s)^{1+\sigma-2\alpha}}{1+\sigma-2\alpha}\notag\\
\leq &C_1 w_{A^{-\alpha} F}(t)s^{(\beta-\sigma+\gamma)-\gamma-1} (t-s)^\gamma,   \label{P20}
\end{align}
where $C_1>0$ is some positive constant.

For $J_{22}(t,s)$, we estimate its norm as follows.
\begin{align}
&\|J_{22}(t,s)\|   \notag\\
&=\Big \|\int_0^{t-s} AS(r)dr  \int_0^s A^{2\alpha} S(s-u)[A^{-\alpha} F(u)-A^{-\alpha} F(s)] du \Big \|\notag\\
&=\Big \|\int_0^{t-s}  \int_0^s A^{1+2\alpha} S(r+s-u)[A^{-\alpha} F(u)-A^{-\alpha} F(s)] du dr\Big \|\notag\\
&\leq \iota_{1+2\alpha} w_{A^{-\alpha} F}(s) \int_0^{t-s}  \int_0^s  (r+s-u)^{-1-2\alpha}   (s-u)^{\sigma}   u^{\beta-\sigma-1}du dr\notag\\
&=\frac{\iota_{1+2\alpha}w_{A^{-\alpha} F}(s)}{2\alpha}   \int_0^s  [(s-u)^{-2\alpha}-(t-u)^{-2\alpha}]   (s-u)^{\sigma}   u^{\beta-\sigma-1}du \notag\\
&=\frac{\iota_{1+2\alpha}w_{A^{-\alpha} F}(s)}{2\alpha}   \int_0^s  [(t-u)^{2\alpha}-(s-u)^{2\alpha}] (t-u)^{-2\alpha}  (s-u)^{\sigma-2\alpha} \notag\\
&\hspace{4cm} \times  u^{\beta-\sigma-1}du \notag\\
&\leq \frac{\iota_{1+2\alpha}(t-s)^{2\alpha} w_{A^{-\alpha} F}(s)}{2\alpha}   \int_0^s   (t-u)^{-2\alpha}  (s-u)^{\sigma-2\alpha}   u^{\beta-\sigma-1}du \notag\\
&= \frac{\iota_{1+2\alpha} w_{A^{-\alpha} F}(s)}{2\alpha}  (t-s)^{2\alpha} \int_0^s   (t-s+u)^{-2\alpha}  u^{\sigma-2\alpha}   (s-u)^{\beta-\sigma-1}du,    \label{E21}
\end{align}
here we used the inequality $(t-u)^{2\alpha}-(s-u)^{2\alpha}\leq (t-s)^{2\alpha}.$
We have
\begin{align*}
&(t-s)^{2\alpha}\int_{\frac{s}{2}}^s   (t-s+u)^{-2\alpha}  u^{\sigma-2\alpha}   (s-u)^{\beta-\sigma-1}du\\
=& (t-s)^{\gamma}\int_{\frac{s}{2}}^s   (t-s)^{2\alpha-\gamma} (t-s+u)^{-2\alpha}  u^{1+\sigma-2\alpha} u^{-1}   (s-u)^{\beta-\sigma-1}du\\
\leq & 2(t-s)^{\gamma} s^{-1}\int_{\frac{s}{2}}^s  [ (t-s)^{2\alpha-\gamma} (t-s+u)^{-2\alpha}  u^{1+\sigma-2\alpha}]  (s-u)^{\beta-\sigma-1}du.
\end{align*}
Since there exists $C_2>0$ such that  for every $\frac{s}{2}\leq u\leq s$
\begin{align*}
(t-s)^{2\alpha-\gamma}& (t-s+u)^{-2\alpha}  u^{1+\sigma-2\alpha}\\
&=\Big(\frac{t-s}{t-s+u}\Big)^{2\alpha-\gamma}  \Big(\frac{u}{t-s+u}\Big)^{\gamma} u^{1+\sigma-2\alpha-\gamma} \leq C_2, 
\end{align*}
we obtain 
\begin{align}
&(t-s)^{2\alpha}\int_{\frac{s}{2}}^s   (t-s+u)^{-2\alpha}  u^{\sigma-2\alpha}   (s-u)^{\beta-\sigma-1}du      \notag\\
\leq & 2C_2(t-s)^{\gamma} s^{-1}\int_0^s  (s-u)^{\beta-\sigma-1}du  \notag\\
= & \frac{2C_2(t-s)^{\gamma} s^{(\beta-\sigma+\gamma)-\gamma-1}}{\beta-\sigma}.   \label{E23}
\end{align}
Meanwhile,
\begin{align}
&(t-s)^{2\alpha}\int_0^{\frac{s}{2}}   (t-s+u)^{-2\alpha}  u^{\sigma-2\alpha}   (s-u)^{\beta-\sigma-1}du            \notag \\
\leq &2^{1-\beta+\sigma} s^{\beta-\sigma-1} (t-s)^{2\alpha}\int_0^{\frac{s}{2}}   (t-s+u)^{-2\alpha}  u^{\sigma-2\alpha}  du            \notag \\
\leq &2^{1-\beta+\sigma} s^{\beta-\sigma-1} (t-s)^{1+\sigma-2\alpha}\int_0^\infty   (1+r)^{-2\alpha}  r^{\sigma-2\alpha}  dr            \notag \\
= &2^{1-\beta+\sigma}  (t-s)^{1+\sigma-2\alpha-\gamma}\int_0^\infty   (1+r)^{-2\alpha}  r^{\sigma-2\alpha}  dr  s^{\beta-\sigma-1} (t-s)^\gamma            \notag \\
\leq &2^{1-\beta+\sigma}  T^{1+\sigma-2\alpha-\gamma}\int_0^\infty   (1+r)^{-2\alpha}  r^{\sigma-2\alpha}  dr  s^{\beta-\sigma-1} (t-s)^\gamma            \notag \\
\leq & C_3s^{\beta-\sigma-1} (t-s)^\gamma,   \label{E25}
\end{align}
where $$C_3=2^{1-\beta+\sigma}  T^{1+\sigma-2\alpha-\gamma}\int_0^\infty   (1+r)^{-2\alpha}  r^{\sigma-2\alpha}  dr <\infty \hspace{0.4cm} \text {(since }  1+\sigma-4\alpha<0). $$
Taking the sum of \eqref{E23} and \eqref{E25} and substituting it for the integral in the right-hand side of  \eqref{E21}, we gain 
\begin{equation} \label{P21}
\|J_{22}(t,s)\|\leq \Big(\frac{2C_2}{\beta-\sigma}+C_3\Big) \frac{\iota_{1+2\alpha}}{2\alpha}       w_{A^{-\alpha} F}(s)s^{(\beta-\sigma+\gamma)-\gamma-1}  (t-s)^{\gamma}. 
\end{equation}

For $J_{23}(t,s)$, by using \eqref{P5} 
\begin{align*}
\|J_{23}(t,s)\|&= \|A^{2\alpha-1} [S(t-s)-S(t)]  [A^{-\alpha} F(s)-A^{-\alpha} F(t)]\| \\
&\leq  \|A^{2\alpha-1} [S(t-s)-S(t)]\| (t-s)^{\sigma -\gamma}w_{A^{-\alpha} F}(t)s^{\beta-\sigma-1}   (t-s)^{\gamma}.
\end{align*}
Then it is  clear that there exists $C_4>0$ such that
\begin{equation} \label{P22} 
\|J_{23}(t,s)\| \leq C_4 w_{A^{-\alpha} F}(t)s^{(\beta-\sigma+\gamma)-\gamma-1} (t-s)^\gamma. 
\end{equation}
Thanks to \eqref{P19}, \eqref{E19},  \eqref{P20}, \eqref{P21} and \eqref{P22}, we conclude that 
$$J_2\in \mathcal F^{\beta-\sigma+\gamma,\gamma}((0,T];H).$$

{\it Proof for $J_3$}. Since $t^{1-\beta} A^{-\alpha}F(t)$ has a limit as $t\to 0$, 
\begin{equation} \label{E27} 
\lim_{t\to 0} t^{1-\beta}J_3(t)=\lim_{t\to 0} A^{2\alpha-1}[I-S(t)]t^{1-\beta} A^{-\alpha}F(t)=0.
\end{equation}
 We next write
\begin{align}
J_3(t)-J_3(s)=&A^{2\alpha-1}[I-S(t)][A^{-\alpha}F(t)-A^{-\alpha}F(s)]   \notag\\
&+A^{2\alpha-1}[I-S(t-s)]S(s)A^{-\alpha}F(s).  \label{E29} 
\end{align}

Let us give estimate for the norm of the first term in the right-hand side of the latter equality. Due to \eqref{P5}, \eqref{P8} and \eqref{P9}  there exists $C_5>0$ such that 
\begin{align}
&\|A^{2\alpha-1}[I-S(t)][A^{-\alpha}F(t)-A^{-\alpha}F(s)]\| \notag\\
&\leq \|A^{2\alpha-1}[I-S(t)]\| w_{A^{-\alpha}F}(t) s^{\beta-\sigma-1} (t-s)^\sigma \notag\\
&\leq C_5 w_{A^{-\alpha}F}(t) s^{(\beta-\sigma+\gamma)-\gamma-1} (t-s)^\gamma.  \label{P23}
\end{align}

For the norm of the second term, we have
\begin{align*}
&\|A^{2\alpha-1}[S(t-s)-I]S(s)A^{-\alpha}F(s)\|\\
&\leq \|[S(t-s)-I]A^{2\alpha-\sigma-1}\|  s^{\beta-\sigma-1} \|s^\sigma A^\sigma  S(s) s^{1-\beta}A^{-\alpha}F(s)\|\\
&\leq \Big\|\int_0^{t-s} A^{2\alpha-\sigma} S(r) dr\Big\| s^{\beta-\sigma-1} \|  s^\sigma A^\sigma S(s) s^{1-\beta}A^{-\alpha}F(s)\|.
\end{align*}
If $2\alpha\geq \sigma$ then
\begin{align*}
&\|A^{2\alpha-1}[S(t-s)-I]S(s)A^{-\alpha}F(s)\|\\
&\leq \int_0^{t-s}  \iota_{2\alpha-\sigma} r^{\sigma-2\alpha}dr s^{\beta-\sigma-1} \| s^\sigma A^\sigma  S(s) s^{1-\beta}A^{-\alpha}F(s)\|\\
&= \frac{\iota_{2\alpha-\sigma}}{1-2\alpha+\sigma}  (t-s)^{1-2\alpha+\sigma-\gamma } (t-s)^\gamma s^{\beta-\sigma-1} \| s^\sigma A^\sigma  S(s) s^{1-\beta}A^{-\alpha}F(s)\|\\
&\leq \frac{\iota_{2\alpha-\sigma}T^{1-2\alpha+\sigma-\gamma }}{1-2\alpha+\sigma}   (t-s)^\gamma s^{(\beta-\sigma+\gamma) -\gamma-1} \| s^\sigma A^\sigma  S(s) s^{1-\beta}A^{-\alpha}F(s)\|.
\end{align*}
Meanwhile, if $2\alpha < \sigma$ then by using \eqref{P7} and \eqref{P9} 
\begin{align*}
&\|A^{2\alpha-1}[S(t-s)-I]S(s)A^{-\alpha}F(s)\|\\
&\leq \int_0^{t-s}  \upsilon_{\sigma-2\alpha} \iota_0 dr s^{\beta-\sigma-1} \| s^\sigma A^\sigma  S(s) s^{1-\beta}A^{-\alpha}F(s)\|\\
&=\upsilon_{\sigma-2\alpha} \iota_0 (t-s)^{1-\gamma} (t-s)^\gamma  s^{(\beta-\sigma+\gamma) -\gamma-1}  \| s^\sigma A^\sigma  S(s) s^{1-\beta}A^{-\alpha}F(s)\|\\
&\leq \upsilon_{\sigma-2\alpha} \iota_0 T^{1-\gamma} (t-s)^\gamma  s^{(\beta-\sigma+\gamma) -\gamma-1}  \| s^\sigma A^\sigma  S(s) s^{1-\beta}A^{-\alpha}F(s)\|.
\end{align*}
Hence, there exists $C_6>0$ such that
\begin{align}
&\|A^{2\alpha-1}[S(t-s)-I]S(s)A^{-\alpha}F(s)\| \notag\\
&\leq C_6  (t-s)^\gamma  s^{(\beta-\sigma+\gamma) -\gamma-1}  \| s^\sigma A^\sigma  S(s) s^{1-\beta}A^{-\alpha}F(s)\|. \label{P24}
\end{align}

On the other hand, since $ s^{1-\beta}A^{-\alpha}F(s)$ has a limit as $s\to 0$, by \eqref{P10} 
$$\lim_{s\to 0} \| s^\sigma A^\sigma  S(s) s^{1-\beta}A^{-\alpha}F(s)\|=0.$$
According to  \eqref{E27}, \eqref{E29},  \eqref{P23} and \eqref{P24},  $J_3$ is then verified to be in $\mathcal F^{\beta-\sigma+\gamma,\gamma}((0,T];H).$

We have thus proved that $A^{\alpha} X$ belongs to $ \mathcal F^{\beta-\sigma+\gamma,\gamma}((0,T];H)$ for every $0\leq \gamma \leq \sigma$. It completes the proof of the theorem.
\end{proof}
\begin{proof}[Proof of Theorem \ref{theorem2}]
Throughout this proof, the notation $C$ will stand for a universal constant which is determined in each occurrence by the exponents and by constants $\iota_\theta, \upsilon_\theta \,(\theta\geq 0)$ in a specific way. We will use  the following property of stochastic integrals (see \cite{prato} for the proof). 
\begin{lemma} \label{lemma1.5}
Let  $ \phi $ be a measurable  function from $ ([0,T],\mathcal B([0,T]))$ to $(L_2(U;H),$ $\mathcal B(L_2(U;H)))$  satisfying  the condition
$$\int_0^t \|\phi(s)\|_{L_2(U;H)}^2 ds<\infty, \hspace{2cm}  t\in [0,T]. $$
Then the stochastic integral $\int_0^t \phi(s) dW(s)$    is well-defined  and is continuous in $ t\in [0,T]. $  Furthermore, 
$$ \mathbb E\Big\| \int_0^t \phi(s) dW(s)\Big\|^2=\int_0^t\|\phi(s)\|_{L_2(U;H)}^2 ds.$$
In fact, the integral is a  continuous square integrable martingale on $[0,T]$. Here and then, the continuity means  that almost surely  trajectories of the stochastic process are continuous in time.
\end{lemma}
We divide the proof  into several steps.

{\bf Step 1}. Let us observe that there exists a unique continuous mild solution of \eqref{P12}, which satisfies the estimate \eqref{P13}. Indeed, due to \eqref{P5} and \eqref{P8},   we have
\begin{align}
&\int_0^t\|S(t-s)G(s)\|_{L_2(U;H)}^2 ds \notag\\
&\leq \int_0^t\|  S(t-s)\|^2 \|G(s)\|_{L_2(U;H)}^2 ds\notag\\
& \leq  \int_0^t \iota_0^2   \|G\|_{\mathcal F^{\beta,\sigma}((0,T];L_2(U;H))}^2 s^{2(\beta-1)}ds\notag\\
&=\frac{\iota_0^2   \|G\|_{\mathcal F^{\beta,\sigma}((0,T];L_2(U;H))}^2 t^{2\beta-1}}{2\beta-1}<\infty, \hspace{2cm}  t\in [0,T]. \label{P25}
\end{align}
Therefore,  by using Lemma \ref{lemma1.5},  it is easy to see that the stochastic convolution 
$$W_G(t)=\int_0^t S(t-s)G(s)dW(s)$$
 is continuous in $t\in [0,T]$. Thus,
$X(t)=S(t) \xi+ W_G(t)$ is a unique continuous mild solution of \eqref{P12}.  Furthermore, by \eqref{P8}, \eqref{P9} and 
\eqref{P25},
\begin{align*}
\mathbb E\|X(t)\|&\leq \mathbb E\|S(t)\xi\|+\mathbb E\|W_G(t)\|\\
&\leq \mathbb E\|S(t)\xi\|+ \sqrt{\mathbb E\|W_G(t)\|^2}\\
&= \mathbb E\|A^{-\beta}S(t) A^\beta \xi\|+  \sqrt{\int_0^t \|S(t-s)G(s)\|_{L_2(U;H)}^2ds}\\
&\leq C[\mathbb E\|A^\beta\xi\|+\|G\|_{\mathcal F^{\beta, \sigma} ((0,T];L_2(U;H))} t^{\beta-\frac{1}{2}}], \hspace{1cm} t\in [0,T].
 \end{align*}
 
{\bf Step 2}.   Let us show that  for every $\nu\in (0,\frac{1}{2}),$  $X\in \mathcal C((0,T];\mathcal D(A^\nu))$ a.s.. 
Indeed, we have
\begin{align*}
&\int_0^t\|A^\nu S(t-s)G(s)\|_{L_2(U;H)}^2 ds\\
&\leq \int_0^t\| A^\nu  S(t-s)\|^2 \|G(s)\|_{L_2(U;H)}^2 ds\\
& \leq  \int_0^t \iota_\nu^2   \|G\|_{\mathcal F^{\beta,\sigma}((0,T];L_2(U;H))}^2 (t-s)^{-2\nu} s^{2(\beta-1)}ds\\
&=\iota_\nu^2   \|G\|_{\mathcal F^{\beta,\sigma}((0,T];L_2(U;H))}^2 B(2\beta-1,1-2\nu)t^{2(\beta-\nu)-1}\\
&<\infty, \hspace{3cm}  t\in (0,T].
\end{align*}
Lemma \ref{lemma1.5} then provides that the stochastic integral $\int_0^tA^\nu S(t-s)G(s)dW(s)$  is continuous on $(0,T]$. Since $A^\nu$ is closed, we obtain 
$$A^\nu W_G(t)=\int_0^tA^\nu S(t-s)G(s)dW(s).$$
Hence $A^\nu W_G$ is continuous on $(0,T]$.
On the other hand,  since $A^\nu S(t)\xi=A^{\nu-\beta} S(t) A^\beta \xi,$ $A^\nu S(\cdot)\xi$ is continuous on $[0,T]$. In this way, we conclude that 
$A^\nu X(t)=A^\nu S(t)\xi+ A^\nu W_G(t)$ is continuous in $t\in (0,T]$, i.e. for every $\nu\in (0,\frac{1}{2}),$ $X\in \mathcal C((0,T];\mathcal D(A^\nu))$ a.s.. 

{\bf Step 3}.   
Let us verify that for every $\alpha_1\in (0, \frac{1}{2}-\sigma]$
\begin{align*}
\mathbb E\|A^{\alpha_1} W_G(t)-A^{\alpha_1} W_G(s)\|^2 \leq & m(t)^2 s^{2(\beta-\sigma-1)} (t-s)^{2\sigma}, \hspace{1cm} 0< s\leq t\leq T 
\end{align*}
where $m(t)$ is some increasing function    on $(0,T]$ such that $\lim_{t\to 0} m(t)=0.$

From the expression
$$A^{\alpha_1} W_G(t)=\int_0^t A^{\alpha_1} S(t-r)[G(r)-G(t)]dW(r)+\int_0^t A^{\alpha_1}  S(t-r) G(t)dW(r),$$
we have
\begin{align*}
&A^{\alpha_1} W_G(t)-A^{\alpha_1} W_G(s)\\
=&\int_s^t A^{\alpha_1} S(t-r)[G(r)-G(t)]dW(r)\\
&+\int_0^sA^{\alpha_1} S(t-r)[G(r)-G(t)]dW(r)\\
&-\int_0^s A^{\alpha_1} S(s-r)[G(r)-G(s)]dW(r)\\
&+\int_s^t A^{\alpha_1} S(t-r)G(t)dW(r)+\int_0^s A^{\alpha_1} S(t-r)G(t)dW(r)\\
&-\int_0^s A^{\alpha_1} S(s-r)G(s)dW(r)\\
=&\int_s^t A^{\alpha_1} S(t-r)[G(r)-G(t)]dW(r)\\
&+\int_0^s A^{\alpha_1} S(t-s)S(s-r)[G(r)-G(s)+G(s)-G(t)]dW(r)\\
&-\int_0^s A^{\alpha_1} S(s-r)[G(r)-G(s)]dW(r)+\int_s^t A^{\alpha_1} S(t-r)G(t)dW(r)\\
&+\int_0^s [A^{\alpha_1} S(t-r)G(t)-A^{\alpha_1} S(s-r)G(s)]dW(r)\\
=&\int_s^t A^{\alpha_1} S(t-r)[G(r)-G(t)]dW(r)\\
&+\int_0^s [S(t-s)-I]A^{\alpha_1} S(s-r)[G(r)-G(s)]dW(r)\\
&+\int_0^s A^{\alpha_1} S(t-r)[G(s)-G(t)]dW(r)+\int_s^t A^{\alpha_1} S(t-r)G(t)dW(r)\\
&+\int_0^s A^{\alpha_1} S(s-r)[G(t)-G(s)]dW(r)\\
&+\int_0^s A^{\alpha_1} [S(t-r)-S(s-r)]G(t)dW(r)\\
=&K_1+K_2+K_3+K_4+K_5+K_6.
\end{align*}
Let us give estimates for the expectation of the square  of norm of $K_i \, (i=1,\dots,6)$. For $K_1$, thanks to  \eqref{P5} and \eqref{P7}  we have
\begin{align*}
\mathbb E\|K_1\|^2 = &\int_s^t \|A^{\alpha_1} S(t-r)[G(r)-G(t)]\|_{L_2(U;H)}^2dr\\
\leq &\int_s^t \|A^{\alpha_1} S(t-r)\|^2 \|G(r)-G(t)\|_{L_2(U;H)}^2dr\\
\leq &\iota_{{\alpha_1}}^2 w_{G}(t)^2  \int_s^t (t-r)^{2(\sigma-\alpha_1)} r^{2(\beta-\sigma-1)}dr\\
\leq &\iota_{\alpha_1}^2 w_{G}(t)^2  s^{2(\beta-\sigma-1)}  \int_s^t (t-r)^{2(\sigma-\alpha_1)}dr\\
= &\iota_{\alpha_1}^2  w_{G}(t)^2  s^{2(\beta-\sigma-1)}   \frac{(t-s)^{1+2(\sigma-\alpha_1)}}{1+2(\sigma-\alpha_1)}\\
\leq &C   w_{G}(t)^2 s^{2(\beta-\sigma-1)}(t-s)^{2\sigma},
\end{align*}
where
$$w_{G}(t)=\sup_{0\leq s\leq t}\frac{s^{1-\beta+\sigma}\|G(t)-G(s)\|_{L_2(U;H)}}{(t-s)^\sigma}.$$
For $K_2$  we have
\begin{align*}
\mathbb E\|K_2\|^2 = & \int_0^s \Big\|\int_0^{t-s} AS(\rho) d\rho A^{\alpha_1}  S(s-r) [G(r)-G(s)]\Big\|_{L_2(U;H)}^2 dr\\
 \leq & \int_0^s \Big\|\int_0^{t-s} A^{1-\sigma} S(\rho) d\rho\Big\|^2 \|A^{\alpha_1+\sigma} S(s-r)\|^2\\
& \hspace{2cm} \times \|G(r)-G(s)\|_{L_2(U;H)}^2 dr\\
\leq & \iota_{1-\sigma}^2\iota_{\alpha_1+\sigma}^2 w_{G}(s)^2\\
&\times \int_0^s \Big(\int_0^{t-s} \rho^{\sigma-1}d\rho\Big)^2  (s-r)^{-2(\alpha_1+\sigma)}  (s-r)^{2\sigma} r^{2(\beta-\sigma-1)}dr\\
\leq  &C w_{G}(s)^2 \int_0^s (t-s)^{2\sigma}  (s-r)^{-2\alpha_1}   r^{2(\beta-\sigma-1)}dr\\
= & C  B(2\beta-2\sigma-1,1-2\alpha_1)s^{2(\beta-\sigma-\alpha_1)-1} w_{G}(s)^2 (t-s)^{2\sigma}\\
\leq & C w_{G}(s)^2  s^{2(\beta-\sigma-1)} (t-s)^{2\sigma}.
\end{align*}
For $K_3$, 
\begin{align*}
\mathbb E\|K_3\|^2 = & \int_0^s \|A^{\alpha_1} S(t-r) [G(s)-G(t)]\|_{L_2(U;H)}^2dr\\
\leq & \int_0^s \|A^{\alpha_1} S(t-r)\|^2 \|G(s)-G(t)\|_{L_2(U;H)}^2dr\\
 \leq & C \int_0^s (t-r)^{-2{\alpha_1}}dr  w_{G}(t)^2  (t-s)^{2\sigma} s^{2(\beta-\sigma-1)}\\
 \leq & C  [t^{1-2\alpha_1}-(t-s)^{1-2\alpha_1} ] w_{G}(t)^2  (t-s)^{2\sigma} s^{2(\beta-\sigma-1)}\\
 \leq & C w_{G}(t)^2  s^{2(\beta-\sigma-1)}   (t-s)^{2\sigma}.
\end{align*}
For $K_4$, 
\begin{align*}
\mathbb E\|K_4\|^2= & \int_s^t \|A^\alpha_1 S(t-r) G(t)\|_{L_2(U;H)}^2dr\\
\leq & \int_s^t \|A^{\alpha_1} S(t-r)\|^2 \|G(t)\|_{L_2(U;H)}^2dr\\
\leq & C \int_s^t (t-r)^{-2\alpha_1}dr \|G\|_{\mathcal F^{\beta,\sigma}((0,T];L_2(U;H)}^2 t^{2(\beta-1)} \\
\leq & C t^{2(\beta-1)}  (t-s)^{1-2\alpha_1}\\
= & C t^{2\sigma} (t-s)^{1-2\alpha_1-2\sigma} t^{2(\beta-\sigma-1)}  (t-s)^{2\sigma}\\
\leq  & C t^{2\sigma}  T^{1-2\alpha_1-2\sigma} s^{2(\beta-\sigma-1)} (t-s)^{2\sigma}\\
\leq  & C t^{2\sigma} s^{2(\beta-\sigma-1)}  (t-s)^{2\sigma}.\end{align*}
For $K_5$, 
\begin{align*}
\mathbb E\|K_5\|^2 =& \int_0^s \|A^{\alpha_1} S(s-r)[G(t)-G(s)]\|_{L_2(U;H)}^2dr\\
\leq & \int_0^s \|A^{\alpha_1} S(s-r)\|^2 \|G(t)-G(s)\|^2dr\\
 \leq &C \int_0^s (s-r)^{-2\alpha_1}dr w_{G}(t)^2 s^{2(\beta-\sigma-1)} (t-s)^{2\sigma}\\
 \leq &C s^{2(\beta-\sigma-\alpha_1)-1} w_{G}(t)^2  (t-s)^{2\sigma}\\
 \leq &C w_{G}(t)^2 s^{2(\beta-\sigma-1)}   (t-s)^{2\sigma}.
\end{align*}
Finally, for $K_6$ we have
\begin{align*}
\mathbb E\|K_6\|^2=& \int_0^s \|A^{\alpha_1} [S(t-r)-S(s-r)]G(t)\|_{L_2(U;H)}^2 dr\\
\leq & \int_0^s \|A^{\alpha_1+\sigma}S(s-r)\|^2 \|S(t-s)-I]A^{-\sigma}\|^2 \|G(t)\|_{L_2(U;H)}^2 dr\\
= & \int_0^s \|A^{\alpha_1+\sigma}S(s-r)\|^2 \Big\|\int_0^{t-s} A^{1-\sigma} S(\rho)d\rho\Big\|^2  dr\|G(t)\|_{L_2(U;H)}^2\\
\leq & \iota_{\alpha_1+\sigma}^2 \iota_{1-\sigma}^2   \int_0^s (s-r)^{-2(\alpha_1+\sigma)} dr \Big[\int_0^{t-s} \rho^{\sigma-1}d\rho\Big]^2 \\
&\times \|G\|_{\mathcal F^{\beta,\sigma}((0,T]; L_2(U;H))}^2 t^{2(\beta-1)} \\
\leq &C s^{1-2(\alpha_1+\sigma)}t^{2\sigma} t^{2(\beta-\sigma-1)} (t-s)^{2\sigma}\\
\leq &C T^{1-2(\alpha_1+\sigma)} t^{2\sigma} s^{2(\beta-\sigma-1)} (t-s)^{2\sigma}\\
\leq & Ct^{2\sigma} s^{2(\beta-\sigma-1)} (t-s)^{2\sigma}.
\end{align*}
In this way, we conclude that 
\begin{align*}
\mathbb E\|A^{\alpha_1} W_G(t)-A^{\alpha_1} W_G(s)\|^2&\leq  6\sum_{i=1}^6 \mathbb E\|K_i\|^2\\
 &\leq  m(t)^2 s^{2(\beta-\sigma-1)} (t-s)^{2\sigma},  
\end{align*}
where $m(\cdot)$ is some increasing function   on $(0,T]$ such that $\lim_{t\to 0} m(t)=0$.

{\bf Step 4}. Let us verify that  for  $ 0<\gamma<\sigma,$ $0<\alpha_1\leq \frac{1}{2}-\sigma$ and $\epsilon\in (0,T]$
$$A^{\alpha_1} X\in  \mathcal C^\gamma([\epsilon,T];H) \hspace{1cm}\text{ a.s.},$$
and 
$$\mathbb E \|A^{\alpha_1} X\|\in  \mathcal F^{\beta,\sigma}((0,T];\mathbb R). $$
On  account of  {\bf Step 2}, $A^{\alpha_1}  W_G$ is a Gaussian process. Using the estimate in {\bf Step 3} we  apply the Kolmogorov continuity theorem to the process  $A^{\alpha_1}  W_G.$ We then obtain 
$$A^{\alpha_1}  W_G\in \mathcal C^\gamma([\epsilon,T];H).$$ 
Furthermore, it has already verified in Theorem \ref{theorem1} that 
$$J_1=A^{\alpha_1} S(t) \xi\in  \mathcal F^{\beta,\gamma}((0,T]; H)$$
 (this result holds true for every $\alpha_1\in (0,1)).$ 
With a remark that
$$ \mathcal F^{\beta,\gamma}((0,T];H) \subset \mathcal C^\gamma([\epsilon,T];H), \hspace{1cm} 0<\epsilon\leq T,$$
we arrive at the first statement, i.e.
$$A^{\alpha_1} X=J_1+A^{\alpha_1}  W_G\in \mathcal C^\gamma([\epsilon,T];H).$$

We shall show the second one. Due to the estimate in {\bf Step 3}, 
\begin{align*}
 [\mathbb E\|A^{\alpha_1} W_G(t)-A^{\alpha_1} W_G(s)\|]^2 &\leq \mathbb E\|A^{\alpha_1} W_G(t)-A^{\alpha_1} W_G(s)\|^2\\
 &\leq  m(t)^2 s^{2(\beta-\sigma-1)} (t-s)^{2\sigma}.
\end{align*}
Hence,
\begin{equation} \label{P26}
\frac{ s^{1-\beta+\sigma} \mathbb E\|A^{\alpha_1} W_G(t)-A^{\alpha_1} W_G(s)\| }{ (t-s)^{\sigma}} \leq m(t).
\end{equation} 
Furthermore, by choosing $\gamma=\sigma$ in  the estimates \eqref{P16}, \eqref{P17} and \eqref{P18},  which hold true for every ${\alpha_1}\in (0,1)$, we obtain 
\begin{align*}
&\frac{s^{1-\beta+\sigma} \|A^{\alpha_1}[S(t)-S(s)] \xi\|}{(t-s)^\sigma} \leq  C \|A^\beta \xi\| \max\{(t-s)^{1-\alpha_1},(t-s)^{1-\sigma}\}.  
\end{align*}
Taking expectation of both the hand sides of the above estimate, it  follows that
\begin{align}
&\frac{s^{1-\beta+\sigma}\mathbb E \|J_1(t)-J_1(s)\|}{(t-s)^\sigma} \leq  C \mathbb E \|A^\beta \xi\| \max\{t^{1-\alpha_1},t^{1-\sigma}\}.  \label{P27}
\end{align}
Combining \eqref{P26} and \eqref{P27} yields that
\begin{align*} 
&\frac{ s^{1-\beta+\sigma} |\mathbb E\|A^{\alpha_1} X(t)\|-\mathbb E\|A^{\alpha_1} X(s)\| |}{ (t-s)^{\sigma}} \notag\\
&\leq \frac{ s^{1-\beta+\sigma} \mathbb E\|A^{\alpha_1} X(t)-A^{\alpha_1} X(s)\| }{ (t-s)^{\sigma}} \notag\\
&= \frac{ s^{1-\beta+\sigma} \mathbb E\|[A^{\alpha_1} W_G(t)-A^{\alpha_1} W_G(s)] +[J_1(t)-J_1(s)]\| }{ (t-s)^{\sigma}} \notag\\
& \leq \frac{ s^{1-\beta+\sigma} [\mathbb E\|A^{\alpha_1} W_G(t)-A^{\alpha_1} W_G(s)\|  + \mathbb E \|J_1(t)-J_1(s)\|]}{ (t-s)^{\sigma}}\notag\\
& \leq m(t)+  C \mathbb E\|A^\beta \xi\| \max\{t^{1-\alpha_1},t^{1-\sigma}\}.  
\end{align*}
Therefore,
\begin{equation}  \label{P28}
\sup_{0\leq s<t\leq T} \frac{ s^{1-\beta+\sigma} |\mathbb E\|A^{\alpha_1} X(t)\|-\mathbb E\|A^{\alpha_1} X(s)\| |}{ (t-s)^{\sigma}}<\infty
\end{equation}
and 
\begin{equation}  \label{P29}
\lim_{t\to 0} \sup_{0\leq s<t} \frac{ s^{1-\beta+\sigma} |\mathbb E\|A^{\alpha_1} X(t)\|-\mathbb E\|A^{\alpha_1} X(s)\| |}{ (t-s)^{\sigma}}=0.
\end{equation}

On the other hand, by using the equality 
$$\mathbb E\|A^{\alpha_1} W_G(t)\|^2=\int_0^t \|A^{\alpha_1} S(t-s) G(s)\|_{L_2(U;H)}^2 ds$$
and  the estimate in {\bf Step 2}, we have  
\begin{align*}
t^{1-\beta} \mathbb E\|A^{\alpha_1} W_G(t)\| & \leq t^{1-\beta} \sqrt{\mathbb E\|A^{\alpha_1} W_G(t)\|^2}\\
&\leq \iota_{\alpha_1}   \|G\|_{\mathcal F^{\beta,\sigma}((0,T];L_2(U;H))} t^{\frac{1}{2}-\alpha_1}\sqrt{ B(2\beta-1,1-2\alpha_1)}. 
\end{align*}
Hence,
$$\lim_{t\to 0} t^{1-\beta}  \mathbb E\|A^{\alpha_1} W_G(t)\|=0.$$
In addition, thanks to  \eqref{P10},  
\begin{align*}
\lim_{t\to 0} t^{1-\beta} \mathbb E\|J(t)\|=\lim_{t\to 0}  \mathbb E\|A^{\alpha_1-1} t^{1-\beta} A^{1-\beta}S(t) A^\beta \xi\|=0.
\end{align*}
In this way, we thus proved that
\begin{equation}  \label{P30}
\lim_{t\to 0} t^{1-\beta} \mathbb E\|A^{\alpha_1} X(t)\|  =\lim_{t\to 0} t^{1-\beta} \mathbb E\|J(t)+A^{\alpha_1} W_G(t)\|=0.
\end{equation}
By \eqref{P28}, \eqref{P29} and \eqref{P30},
the second statement has been verified. We complete the proof of the theorem.
\end{proof}
\begin{proof}[Proof of Theorem \ref{theorem3}]
Clearly, the assumptions of  Theorem \ref{theorem1} and Theorem  \ref{theorem2} are fulfilled. Therefore, 
 it is easy to see that Theorem \ref{theorem3} follows from these two theorems.
\end{proof}
 \section{An application to neurophysiology} \label{section3}
  In this section, we will apply our results to a model arising in neurophysiology, which has been proposed by Walsh \cite{Walsh}.  Nerve cells operate by a mixture of chemical, biological and electrical properties. We shall regard them as long thin cylinders, which act much like electrical cables. Denote by $V(t,x)$ the electrical potential at time $t$ and the point $x$. If we identify such a cylinder with the interval $[0,L]$ then $V$ satisfies a nonlinear equations coupled with a system of ordinary differential equations, called the Hodgkin-Huxley equations.  In some certain ranges of the values of the potential, the equations are approximated by the cable equation:
 \begin{equation} \label{P31}
 \begin{cases}
\frac{\partial V}{\partial t}=\frac{\partial^2 V}{\partial x^2} -V, \\
 V\colon [0,\infty)\times [0,L] \to \mathbb R,  \\
 t\geq 0, x\in [0,L].
 \end{cases}
 \end{equation}
 Let us consider \eqref{P31} with an additional external force, called impulses of current. Neurons receive these impulses of current via synapses on their surface. Denote by $F(t)$ the current arriving at time $t$, then  the electrical potential satisfies the equation: 
  \begin{equation*} 
\begin{cases}
\frac{\partial V}{\partial t}=\frac{\partial^2 V}{\partial x^2} -V +F(t), \\
 V\colon [0,\infty)\times [0,L] \to \mathbb R,  \\
 t\geq 0, x\in [0,L].
 \end{cases}
 \end{equation*}
Suppose now that the current $F$ is perturbed by a space-time white noise $\dot W$, i.e.
$$F(t)\leadsto F(t) + G(t) \dot W.$$
 This leads us to the stochastic partial differential equation:
  \begin{equation} \label{P32}
  \begin{cases}
\begin{aligned}
&\frac{\partial V}{\partial t}(t,x)=\frac{\partial^2 V}{\partial x^2}(t,x) -V(t,x) +F(t) + G(t) \frac{\partial W(t)}{\partial t},&\hspace{0.5cm}  t> 0, x\in [0,L],\\
& \frac{\partial V}{\partial x}(0,t)= \frac{\partial V}{\partial x}(L,t)=0, &\hspace{0.5cm} t>0,\\
&V(0,x)=V_0(x), &\hspace{0.5cm}x\in (0,L),
\end{aligned}
 \end{cases}
 \end{equation}
where $W$ is a  cylindrical   Wiener process defined on a filtered probability space $(\Omega, \mathcal F,\mathcal F_t,\mathbb P),$ $V_0$ is an $\mathcal  F_0$-measurable  random variable. More explanations for the construction of the system \eqref{P32} can be found in  \cite{Walsh}.
 
Let us next construct an abstract formulation for \eqref{P32}. We will handle the equation in the Hilbert space $H=L_2([0,L])$. We assume that   the cylindrical   Wiener process $W$ takes value on a separable Hilbert space $U$, and that $F$ and $G$ are measurable from $([0,\infty),\mathcal B([0,\infty)))$ to $ (H,\mathcal B(H))$ and $(L_2(U;H),\mathcal B(L_2(U;H)), $ respectively. Denote by $A$ the realization of the operator $-\frac{\partial^2}{\partial x^2}+I$
 in $H$ under the Neumann boundary conditions on the boundary $\{0,L\}$. We formulate the equation \eqref{P32} as the Cauchy problem for an abstract stochastic linear equation
\begin{equation} \label{P33}
\begin{cases}
dV+AV=F(t)+G(t)dW(t), \hspace{1cm} 0<t<\infty,\\
V(0)=V_0
\end{cases}
\end{equation}
in $H$. 
\begin{lemma}\label{lemma2}
The operator $A$ is a positive definite self-adjoint operator on $H$ with  domain 
$$\mathcal D(A)=H_N^2([0,L]):=\{u\in H^2([0,L]) \text{ such that } \frac{\partial u}{\partial x}(0)= \frac{\partial u}{\partial x}(L)=0\}.$$ 
 In addition, the domains of fractional powers $A^\theta, 0\leq \theta\leq 1,$ are characterized by 
\begin{equation*}
\mathcal D(A^\theta)=
\begin{cases}
H^{2\theta}([0,L]) \hspace{1cm} \text{ for } 0\leq \theta<\frac{3}{4},\\
H_N^{2\theta}([0,L])  \hspace{1cm} \text{ for } \frac{3}{4}< \theta\leq 1.
\end{cases}
\end{equation*}
\end{lemma}
The lemma \ref{lemma2} is a special case of Theorem 16.7 and Theorem 16.9 in  \cite{yagi}. It then follows that $A$ is a sectorial operator.
 The following theorem follows from Theorem \ref{theorem1}, Theorem \ref{theorem2} and Theorem \ref{theorem3}. It shows global existence and regularity of solutions to \eqref{P33}.
\begin{theorem} \label{theorem5}
Assume that $V_0\in \mathcal D(A^\beta) $ a.s. and $\mathbb E\|A^\beta V_0\|<\infty.$ 
\begin{itemize}
\item [\rm (i)] If  $G\equiv 0$ on  $[0,\infty)$ (i.e. \eqref{P33} is a deterministic equation) and $F$ satisfies {\rm (H3)} for every $T>0,$ then the equation \eqref{P33} has a unique continuous mild solution $V$ on $[0,\infty)$ possessing the regularity:
$$V\in \mathcal C((0,\infty);\mathcal D(A^{1-\alpha})), $$
$$A^{\alpha} V\in \mathcal C([0,\infty);H)\cap \mathcal F^{\beta-\sigma+\gamma,\gamma}((0,T];H), $$
  and  satisfying the estimate
\begin{equation*} 
\|V(t)\|\leq \iota_\alpha B(\beta,1-\alpha) \|A^{-\alpha}F\|_{\mathcal F^{\beta,\sigma}} t^{\beta-\alpha} +\iota_0 \|\xi\|
\end{equation*}
for every $T>0$, $t\in [0,T]$ and $\gamma\in [0, \sigma]$.
\item [\rm (ii)] If  $F\equiv 0$ on  $[0,\infty)$  and  $G$ satisfies {\rm (H4)} for every $T>0,$ then \eqref{P33} has a unique mild solution $V$ on $[0,\infty)$ possessing the regularity:
$$V\in \mathcal C((0,\infty);\mathcal D(A^\nu)), \quad A^{\alpha_1} V\in  \mathcal C^\gamma([\epsilon,T];H) \hspace{1cm}\text{ a.s.}$$
and 
$$\mathbb E \|A^{\alpha_1} V\|\in  \mathcal F^{\beta,\sigma}((0,T];\mathbb R), $$
   and satisfying the estimate
  \begin{equation*} 
  \mathbb E\|V(t)\|\leq C[\mathbb E\|A^\beta\xi\|+\|G\|_{\mathcal F^{\beta, \sigma} ((0,T];L_2(U;H))} t^{\beta-\frac{1}{2}}], \hspace{1cm} t\in [0,T]
  \end{equation*}
  for every $T>0, \nu\in (0,\frac{1}{2}), \alpha_1\in (0, \frac{1}{2}-\sigma],  \gamma\in (0,\sigma)$ and $\epsilon\in (0,T]$, where $C$ is some constant depending only on the exponents and  constants $\iota_\theta, \upsilon_\theta \,(\theta\geq 0)$.
\item [\rm (iii)] If  $F$ and $G$ satisfy {\rm (H3)} and {\rm (H4)}  for every $T>0$ and $\alpha\in (0, \frac{1}{2}-\sigma],$  
then there exists a unique  mild solution of \eqref{P12} possessing the regularity:
$$V\in \mathcal C((0,\infty);\mathcal D(A^\nu)), \quad A^{\alpha} V\in  \mathcal C^\gamma([\epsilon,T];H) \hspace{1cm}\text{ a.s.}$$
and 
$$\mathbb E \|A^{\alpha} V\|\in  \mathcal F^{\beta,\sigma}((0,T];\mathbb R), $$
   and satisfying the estimate
  \begin{equation*} 
  \mathbb E\|V(t)\|\leq C[\mathbb E\|A^\beta\xi\|+\|A^{-\alpha} F\|_{\mathcal F^{\beta,\sigma}}\|G\|_{\mathcal F^{\beta, \sigma} ((0,T];L_2(U;H))} t^{\beta-\frac{1}{2}}] \hspace{0cm} 
  \end{equation*}
  for every $t\in [0,T],$  $T>0, $ $ \nu\in (0,\frac{1}{2}),   \gamma\in (0,\sigma)$ and $\epsilon\in (0,T]$, where $C$ is some constant depending only on the exponents and  constants $\iota_\theta, \upsilon_\theta \,(\theta\geq 0)$.
\end{itemize}
\end{theorem}
\begin{remark}
A result in \cite{Walsh} shows that $V$ is a H\"{o}lder continuous function with exponent $\frac{1}{4}-\epsilon$, for any $\epsilon>0$. We have improved this result by showing that $A^\alpha V$ is H\"{o}lder continuous  with an arbitrary exponent smaller than $\sigma$ (noting that $\sigma$ in Theorem \ref{theorem5} is possibly larger than $\frac{1}{4}$). Using the same framework, we can treat the equation \eqref{P32} in higher dimensions and obtain similar results.
\end{remark}
\section*{Acknowledgements}
A part of this paper was done  while the author was visiting the Faculty of Mathematics, Bielefeld University in January 2015. The author would like to thank Prof. Michael R\"{o}ckner for his support during the stay at the university. Thanks also goes to the referee for the valuable comments.


\begin{thebibliography}{9}
\bibitem {Ball} J. M. Ball, 
      \newblock{Strongly continuous semigroups, weak solutions, and the variation of constants formula}, 
     \textit{Proc. Amer. Math. Soc.} \textbf{63}  (1977),  370-373.  \MR{MR0442748}
\bibitem {Brzezniak} Z. Brze\'{z}niak,
      Stochastic partial differential equations in M-type 2 Banach spaces, 
      \textit{Potential Anal.} \textbf{4}  (1995),  1-45.  \MR{MR1313905}
\bibitem{prato-2} G. Da Prato, P. Grisvard,  
      \newblock{Equations d'\'{e}volution abstraites non lin\'{e}aires de type parabolique},
      \textit{Ann. Mat. Pura Appl.} \textbf{120} (1979),  329-396. \MR{0551075}
\bibitem{prato0} G. Da Prato, S. Kwapien, J. Zabczyk,  
      \newblock{Regularity of solutions of linear stochastic  equations in Hilbert spaces},
      \textit{Stochastic} \textbf{23} (1987),  1-23. \MR{0920798}
\bibitem{prato} G. Da Prato, J.  Zabczyk,   
      \textit{Stochastic Equations in Infinite Dimensions},   
      \newblock{Cambridge}, 1992. \MR{1207136}
\bibitem{Hairer}  M.  Hairer,
An introduction to stochastic PDEs,  \textit{ArXiv e-prints} [arXiv:0907.4178] (2009), 78 pages. 
  \bibitem{Krylov1}  N. V. Krylov, S. V. Lototsky,
 A Sobolev space theory of SPDEs with constant coefficients in a half space,
         \textit{SIAM J. Math. Anal.} {\bf 31} (1999) 19-33. \MR{1720129}
\bibitem{Mikulevicius2}  R. Mikulevicius, H. Pragarauskas, N. Sonnadara,
 On the Cauchy-Dirichlet problem in the half space for parabolic SPDEs in weighted Hoelder spaces,
 \textit{Acta Appl. Math.} {\bf 97} (2007) 129-149.  \MR{2329725}
\bibitem{Pazy} A. Pazy,   
      \textit{Semigroups of Linear Operators and Applications to Partial Differential Equations},  
      \newblock{Springer-Verlag}, 1983. \MR{0710486}
\bibitem{PrevotRockner} C.  Pr\'{e}v\^{o}t, M.  R\"{o}ckner, 
          \textit{A Concise Course on Stochastic Partial Differential Equations}, 
         Lecture Notes in Mathematics  \textbf{1905}, Springer, Berlin, 2007.  \MR{2329435}
\bibitem{Rozovskii} B. L. Rozovskii,
          \textit{Stochastic Evolution Systems. Linear Theory and Applications to Nonlinear Filtering}, 
          Kluwer Academic Publishers Group, Dordrecht, 1990.  \MR{1135324}
\bibitem{Sinestrari}  E. Sinestrari,     
         \newblock{On the abstract Cauchy problem of parabolic type in spaces of continuous functions},  
         \textit{J. Math. Anal. Appl.}  \textbf{107} (1985), 16-66. \MR{0786012}
\bibitem{Ton0} V. T. T\d{a}, A. Yagi,     
Abstract stochastic evolution equations in M-type 2 Banach spaces,        \textit{ArXiv e-prints} [arXiv:1410.0144] (2014), 71 pages. 
\bibitem{Ton1} V. T. T\d{a}, Note on regularity of solutions of abstract linear evolution equations,  \textit{ArXiv e-prints} [arXiv:1505.01252] (2015), 40 pages. 
\bibitem{Ton2} V. T. T\d{a}, A. Yagi,     
Some results on abstract stochastic semilinear evolution  equations,        \textit{ArXiv e-prints} [arXiv:1508.07340v1] (2015), 56 pages. 
\bibitem{Ton3} V. T. T\d{a}, Note on abstract stochastic semilinear evolution equations,  \textit{ArXiv e-prints} [arXiv:1508.07211] (2015), 39 pages. 
\bibitem{Ton4} V. T. T\d{a}, Y. Yamamoto, A. Yagi,      
Strict solutions to stochastic linear evolution equations in M-type 2 Banach spaces,     \textit{ArXiv e-prints} [arXiv:1508.07431v2] (2015), 24 pages.
\bibitem {Walsh} J. B. Walsh,  
\textit{An Introduction to Stochastic Partial Differential Equations}, 
 \'{E}cole d'\'{e}t\'{e} de probabilit\'{e}s de Saint-Flour, XIV-1984, 265-439, Lecture Notes in Mathematics \textbf {1180}, Springer, Berlin, 1986.  \MR{0876085}
\bibitem {yagi0} A. Yagi,   
      \newblock{Fractional powers of operators and evolution equations of parabolic type}, 
      \textit{Proc. Japan Acad. Ser. A Math. Sci.}
      \textbf{64} (1988),  227-230. \MR{0974079}
\bibitem {yagi1} A. Yagi,   
      \newblock{Parabolic evolution equations in which the coefficients are the generators of infinitely differentiable semigroups}, 
       \textit{Funkcial. Ekvac.}
      \textbf{32}  (1989), 107-124. \MR{1006090}
\bibitem {yagi2} A. Yagi,   
      \newblock{Parabolic evolution equations in which the coefficients are the generators of infinitely differentiable semigroups, II}, 
       \textit{Funkcial. Ekvac.}
      \textbf{33}  (1990), 139-150. \MR{1065472}
\bibitem {yagi} A. Yagi,   
      \textit{Abstract Parabolic Evolution Equations and their Applications}, 
      \newblock{Springer-Verlag, Berlin,} 2010. \MR{2573296}
\end{thebibliography}
\end{document}